\documentclass{amsart}
\pdfoutput=1
\usepackage{mathrsfs}
\usepackage{MnSymbol}
\usepackage[mathscr]{eucal}
\usepackage{mathtools}
\usepackage{url, array}
\usepackage[T1]{fontenc}
\usepackage{color}
\usepackage{comment}
\usepackage{tikz-cd}
\usepackage[all]{xy}
\SelectTips {cm}{}

\newcommand{\up}[1]{\textup{\texttt{#1}}}
\DeclareMathOperator{\arr}{arr}
\DeclareMathOperator{\operads}{\up{Op}}
\DeclareMathOperator{\model}{\up{Model}}
\DeclareMathOperator{\cocomplete}{\up{coComplete}}
\DeclareMathOperator{\twist}{\up{Twist}}
\DeclareMathOperator{\chain}{\up{Ch}}
\DeclareMathOperator{\alg}{\up{-alg}}
\DeclareMathOperator{\coalg}{\up{-coalg}}

\newcommand{\fieldk}{\mathbf{k}}
\newcommand{\aug}{\epsilon}

\newcommand{\Cc}{\mathcal{C}}

\newcommand{\Dd}{\mathcal{D}}

\newcommand{\id}{\mathrm{id}}
\newcommand{\Oo}{\mathcal{O}}
\newcommand{\Pp}{\mathcal{P}}

\newcommand{\dg}{}
\newcommand{\suspend}{\mathscr{S}}
\newcommand{\desuspend}{\suspend^{-1}}
\newcommand{\shifted}{\suspend^{-1}}
\newcommand{\ass}{{\mathcal{A}ss}}
\newcommand{\com}{{\mathcal{C}om}}
\newcommand{\lie}{{\mathcal{L}ie}}

\newcommand{\shiftedcom}{\shifted\com}

\newtheorem{thm}{Theorem}[section]
\newtheorem{lemma}[thm]{Lemma}
\newtheorem{cor}[thm]{Corollary}
\newtheorem{prop}[thm]{Proposition}

\theoremstyle{definition}

\newtheorem*{summary*}{Summary}
\newtheorem{defi}[thm]{Definition}
\newtheorem{example}[thm]{Example}

\theoremstyle{remark}
\newtheorem{remark}[thm]{Remark}
\newtheorem*{remark*}{Remark}

\title{Model structures for coalgebras}
\author{Gabriel C. Drummond-Cole}
\thanks{This work was supported by IBS-R003-G1}
\address{Center for Geometry and Physics, Institute for Basic Science (IBS), Pohang 790-784, Republic of Korea}
\author{Joseph Hirsh}
\thanks{The second author was supported by NSF DMS-1304169.}
\address{Department of Mathematics, Massachusetts Institute of Technology (MIT), Cambridge 02139, United States of America}

\begin{document}
\begin{abstract}
Classically, there are two model category structures on coalgebras in the category of chain complexes over a field. In one, the weak equivalences are maps which induce an isomorphism on homology. In the other, the weak equivalences are maps which induce a weak equivalence of algebras under the cobar functor. We unify these two approaches, realizing them as the two extremes of a partially ordered set of model category structures on coalgebras over a cooperad satisfying mild conditions.
\end{abstract}
\maketitle
\section{Introduction}
A natural area of interest for an algebraic homotopy theorist is the homotopy theory of coalgebras. Several different ideas about the homotopy category of coalgebras (say, in some category of chain complexes) exist in the literature. 

Quillen~\cite{Quillen:RHT} gave a model structure on (a subcategory of) commutative coalgebras where the weak equivalences were created by the cobar functor to Lie algebras. Hinich~\cite{Hinich:DGCFS} extended the same structure to a larger category of commutative algebras and Lefevre-Hasegawa~\cite{LefevreHasegawa:SAIC} then extended these methods to (a subcategory of) coassociative coalgebras. More recently, Vallette~\cite{Vallette:HTHA} described these three model structures as examples of a much more general phenomenon related to Koszul duality. That is, a Koszul twisting morphism from a cooperad to a operad induces a bar-cobar adjunction between coalgebras over the cooperad and algebras over the operad. Vallette used this adjunction to lift the (well-known) model category structure on algebras over the operad to a model category structure on coalgebras over the cooperad. In the Vallette model structure, the weak equivalences are the morphisms which become weak equivalences of algebras under the cobar functor. The Vallette model structure provides a model category framework for infinity algebras ($A_\infty$, $L_\infty$, and so on), which can be realized as the fibrant-cofibrant objects in this model category.

On the other hand, if one is interested in the theory of derived coalgebras, it makes more sense to have the weak equivalences be the weak equivalences of chain complexes ({quasi-isomorphisms}) under the forgetful functor. This point of view has also been taken in the literature. Getzler and Goerss~\cite{GetzlerGoerss:MCSDGC} constructed such a model category structure for coassociative coalgebras. Aubry and Chataur~\cite{AubryChataur:CCCMC} proved the existence of such a model category structure for coalgebras over a quasi-cofree cooperad in chain complexes. Smith~\cite{Smith:MCCOC} provided such a structure for more general cooperads.\footnote{Smith works over a field or the integers, using \emph{homotopy equivalence} of chain complexes as his weak equivalences. Over a field this is the same as quasi-isomorphism; over the integers it is not and Smith's structure does not fit into our story.}
Hess and Shipley~\cite{HessShipley:HTCOC} established the existence of a model category structure of this sort in a much more general situation, that of coalgebras over a comonad, not necessarily in a category of chain complexes. Applied to the case where the comonad comes from a cooperad in a category of chain complexes, they give the most general version of this model structure for cooperads in bounded below chain complexes. 

When a cooperad is concentrated in arity one, it is the same thing as a counital coassociative coalgebra. A ``coalgebra'' over such a cooperad is then a comodule over that counital coalgebra. In this special case, the two classical choices of weak equivalences for the category of comodules give rise to what are sometimes called the derived categories of the second and first kind, respectively, which have a long history~\cite{HusemollerMooreStasheff:DHAHS, Positselski:TKDCKDCCC}. See Section~\ref{subsection:comodules} for an example specifically related to this case.

Porta asked a question on \url{mathoverflow.net}~\cite{Porta:MSCC} about model structures on coalgebras. In part, the question asked where the {\em Koszul} property was used in Vallette's construction. It asked whether instead some more general sort of twisting morphism between a cooperad and operad might induce a model category structure on coalgebras.

We answer this question in the affirmative, proving the existence of a model category structure on coalgebras over a cooperad equipped with a twisting morphism to some operad. In these model category structures, the weak equivalences and cofibrations are created by the cobar functor to the category of algebras over the operad. The twisting morphism is not assumed to be a Koszul morphism and the operad is not assumed to be dual to the cooperad.

This construction unifies the two main approaches to defining model category structures on coalgebras over cooperads. In fact, it yields a {\em diagram} of model category structures on the same underlying category of coalgebras. The initial object in this category, using the canonical twisting morphism, is the Vallette model category structure and the terminal object, using the trivial twisting morphism, is the Aubry--Chataur/Hess--Shipley model category structure.

At the level of $\infty$-categories, the former is the $\infty$-category of algebras and the latter the $\infty$-category of coalgebras. Thus our results show that the $\infty$-category of coalgebras is a localization of the $\infty$-category of algebras. Moreover, this localization is filtered by the diagram category of twisting morphisms.

From this point of view, the Koszul bar and cobar functors are equivalences of $\infty$-categories. This means that, composing with localization, at the $\infty$-categorical level, the bar functor from the $\infty$-category of algebras over an operad to the $\infty$-category of coalgebras over the Koszul dual cooperad is {\em left} adjoint to the cobar functor. This reverses the parity of the adjunction used in classical operad theory and recovers the parity that is expected by the work of Lurie~\cite[5.2]{Lurie:HA} and Ayala-Francis~\cite{AyalaFrancis:PKD}.

One caveat is that we work over a field. This does not seem to be necessary. However, in the proof of Theorem~\ref{thm:lift} we use results of Vallette~\cite{Vallette:HTHA}, who works over a field. It is likely that his results, formulated in terms of the underlying cofibrations of the model category on chain complexes, hold over a commutative ring, but in the interest of brevity we have chosen not to address this question in this paper. 

The authors would like to thank Bruno Vallette and Emily Riehl for useful conversations, and acknowledge the multiple helpful suggestions of an anonymous referee.

\section{Summary of results}
\begin{summary*}
Let $\alpha$ be a twisting morphism between a conilpotent coaugmented weight-graded cooperad $\Cc$ and an augmented operad $\Pp$ in the category of chain complexes (either unbounded or bounded below as specified in remark~\ref{rem:bounded}) over a field. If the characteristic is not zero, assume $\Omega\Cc$ and $\Pp$ are $\Sigma$-split.

Then there exists a model category structure on $\Cc$-coalgebras where the weak equivalences and cofibrations are created by the cobar functor $\Omega_\alpha$ to $\Pp$-algebras. This assignment is functorial in both $\Cc$ and $\Pp$.

Given this model structure, the cobar functor $\Omega_\alpha$ is a left Quillen functor, which is a Quillen equivalence if $\alpha$ is Koszul. The converse is true if the characteristic is zero and $\alpha$ respects an additional weight grading on $\Pp$.

The cofibrations for arbitrary $\alpha$ are also created by the forgetful functor to chain complexes, and thus two such model categories related by a morphism of operads are related via Bousfield localization.
\end{summary*}

\section{Model category structures}\label{sec:modelcatstructures}

We assume some familiarity with the theory of operads, more or less following the conventions of Loday--Vallette~\cite{LodayVallette:AO}. We refer the reader to that text for a full history and more references for the various conventions and constructions. As that text is written primarily in characteristic zero, the reader is also invited to refer to~\cite{Fresse:OCCCOHMAOO} for versions of some constructions and theorems in positive characteristic.\footnote{A warning: Fresse's conventions differ from those of~\cite{LodayVallette:AO} and ours in several ways. Aside from extensive notational differences, in~\cite{Fresse:OCCCOHMAOO}, Fresse always uses a connected weight grading that coincides with arity. There is no problem generalizing this to an arbitrary connected weight grading. He also uses a $\Sigma_*$-cofibrancy condition which is implied by our $\Sigma$-splitness assumption.}

In particular, we will make immediate use of the cobar functor $\Omega$ from conilpotent cooperads to augmented operads (defined in~\cite[6.5.2, Section 3]{LodayVallette:AO,Fresse:OCCCOHMAOO} or~\cite[Section 3]{Fresse:OCCCOHMAOO}).

Let $\chain$ denote the model category of unbounded chain complexes over a field $\fieldk$, with weak equivalences given by quasi-isomorphisms and fibrations given by surjections. All operads, cooperads, algebras, and coalgebras are taken over this ground category. So everything is assumed to be ``differential graded.''

Unless otherwise specified, we use the following notation.
\begin{itemize} 
	\item  Let $\Cc$ be a conilpotent coaugmented weight graded cooperad such that $\Omega\Cc$ is $\Sigma$-split. We denote the additional weight by $\Cc^{(i)}$, and further assume that the weight grading is connected, i.e., that $\Cc^{(0)}$ is spanned by the image of the coaugmentation in $\Cc$. 
	\item Let $\Pp$ be a $\Sigma$-split augmented operad. 
\end{itemize}
We shall denote the categories of conilpotent $\Cc$-coalgebras (in this paper all coalgebras are assumed to be conilpotent) and $\Pp$-algebras by $\Cc\coalg$ and $\Pp\alg$. 

The $\Sigma$-split conditions (defined in~\cite[4.2]{Hinich:HAHA}) are used only to ensure the existence of model structures on the categories of $\Pp$-algebras and $\Omega\Cc$-algebras. In particular, if the characteristic of $\fieldk$ is zero, all operads are $\Sigma$-split. 
Our proofs rely on results of Vallette~\cite{Vallette:HTHA} that use the connected weight graded assumptions on $\Cc$. 

Some previous authors have shied away from claiming that the category $\Cc\coalg$ is complete, but this is implied by \cite[A.1]{ChingRiehl:CMCMC}, following the arguments of~\cite{AdamekRosicky:LPAC}. See Lemma~\ref{lemma:locallypresentable}. For the interested reader, Agore~\cite{Agore:LCBHA} gives a recipe for explicit limits in the category of not-necessarily conilpotent coassociative coalgebras. Substituting the cofree conilpotent $\Cc$-coalgebra for the cofree coassociative coalgebra yields limits for nilpotent coalgebras over a more general cooperad.

\begin{remark}\label{rem:bounded}
If it is desirable, we may work instead in categories of chain complexes bounded below as follows. Let $\ell$ and $m$ be integers with $\ell+m\ge 0$. Work with 
\begin{itemize}
	\item the category of operads whose arity $n$ component is concentated in degree at least $\ell(n-1)$
	\item the category of cooperads such that the arity $n$ component of the coaugmentation coideal is concentrated in degree at least $\ell(n-1)+1$ 
	\item the categories of and algebras and coalgebras concentrated in degree at least $m$. 
\end{itemize} All constructions and theorems remain the same. Slightly more care is needed for two proofs (see Remarks~\ref{rem:boundedHinich} and~\ref{rem:boundedVallette}).  The condition on the operad $\Pp$ ensures that the free (unbounded) $\Pp$-algebra on a chain complex concentrated in degree at least $m\ge -\ell$ is itself concentrated in degree at least $m$. The condition on the cooperad has a cognate purpose for coalgebras and is shifted by one because the bar and cobar functors between operads and cooperads employ such a shift.
\end{remark}

\begin{defi}
The category of {\em twisting morphisms} from the fixed cooperad $\Cc$ is the undercategory of $\Omega\Cc$ in $\operads$, the category of $\Sigma$-split augmented operads. We denote this category $\twist_{\Cc}$, and refer to objects in $\twist_{\Cc}$ as \emph{twisting morphisms}. A twisting morphism is called {\em Koszul} if it is a quasi-isomorphism.

\end{defi} 
Note that the identity $\id$ of $\Omega \Cc$ is initial and the augmentation $\aug$ of $\Omega\Cc$ is terminal in the category of twisting morphisms from $\Cc$ ($\id$ is called $\iota$ in~\cite{LodayVallette:AO}). Also, recall that a twisting morphism $\alpha : \Omega \Cc \to \Pp$ induces an adjunction $\Omega_{\alpha} : \dg \Cc \coalg \leftrightarrow \dg \Pp \alg : B_{\alpha}$, with $\Omega_{\alpha}$ the left adjoint (see~\cite[11.2--11.3]{LodayVallette:AO} or~\cite[4.2]{Fresse:OCCCOHMAOO}). 

\begin{defi} \label{model_class}
Given $\alpha$ in $\twist_{\Cc}$ we have the following classes of maps in \dg$\Cc$-coalgebras.
\begin{itemize}
	\item The $\alpha$-\emph{weak equivalences} are maps $X \xrightarrow{f} Y$ whose image under the cobar construction $\Omega_{\alpha}(f): \Omega_{\alpha} X \xrightarrow{\sim} \Omega_{\alpha} Y$ is a quasi-isomorphism of \dg$\Pp$-algebras.
	\item The $\alpha$-\emph{cofibrations} are degree-wise monomorphisms of \dg $\Cc$-coalgebras
	\item The $\alpha$-\emph{fibrations} are those maps $f$ which have the right lifting property with respect to every acyclic $\alpha$-cofibration.  
\end{itemize}
\end{defi}

\begin{thm}(Vallette~\cite[2.1(1)]{Vallette:HTHA}) \label{Vallettemodel}
For any $\alpha$ which is a Koszul twisting morphism (called $\kappa$ in~\cite[7.4.1]{LodayVallette:AO}), the classes of morphisms in Definition \ref{model_class} form a model category structure on \dg $\Cc$-coalgebras. 
\end{thm}
We refer to this model structure as the \emph{Vallette model structure}. Explicitly, this theorem is proven is in characteristic zero but working in arbitrary characteristic does not change the argument.

\begin{defi}
Let $f:\Oo\to\Pp$ be a map of operads. The {\em induction} functor along $f$ from $\Oo\alg$ to $\Pp\alg$ is the left adjoint to the restriction functor along $f$ from $\Pp\alg$ to $\Oo\alg$ given by precomposition with $f$.
\end{defi}
Next, we describe two functorial constructions that employ the induction functor. We will state them in terms of functors to very large categories; a reader worried about set-theoretic issues can rephrase them in terms of assignments that respect composition or alternatively can restrict to the image, which is an ordinary category.
\begin{defi}
The very large category $\cocomplete$ has large cocomplete categories as objects, and left adjoint functors as morphisms.
\end{defi}
\begin{prop} \label{prop:arrowcocomplete}
There is a functor $\mathbf{\Omega}$ from $\twist_{\Cc}$ to the arrow category of $\cocomplete$ which takes $\alpha:\Omega\Cc\to \Pp$ to 
$$\dg\Cc\coalg \xrightarrow{\Omega_\alpha} \dg\Pp\alg$$
and takes  a map $f:\Oo\to\Pp$ under $\Omega\Cc$ to the identity on $\Cc\coalg$ and the induction functor along $f$.
\end{prop}
\begin{proof}
It is easy to check that the right adjoints commute and that composition of morphisms under $\Omega\Cc$ is respected. 
\end{proof}
Diagramatically, using the left adjoints, the commutativity of the left diagram in $\operads$ implies the commutativity of the right diagram in $\cocomplete$.
\begin{center}\begin{tikzcd}
\Omega\Cc
\ar{dr}[swap]{\beta}
\ar{r}{\alpha}
&\Oo\ar{d}{f}
\\
&\Pp
\end{tikzcd}
\qquad
\begin{tikzcd}
\Cc\coalg
\ar{drr}[swap]{\Omega_\beta}\ar{rr}{\Omega_\alpha} &&
\Oo\alg
\ar{d}{\text{induction along }f}
\\&&
\Pp\alg.
\end{tikzcd}
\end{center}
\begin{defi}
The very large category $\model$ has objects large model categories, and morphisms left Quillen functors.
\end{defi}
Note that there is a forgetful functor from $\model$ to $\cocomplete$.

\begin{thm} \label{operadmodel} (Hinich)
There is a functor from $\operads$ to $\model$ which takes $\Pp$ to its model category of algebras and a map of operads to the induction functor along it. The image of $\Pp$ under this functor is a cofibrantly generated model category with fibrations and weak equivalences created by the forgetful functor to $\chain$. This functor takes $\Sigma$-split quasi-isomorphisms of operads to left Quillen equivalences.
\end{thm}
\begin{proof}
The existence of such a model category on $\Pp\up{-alg}$ is shown in~\cite[Section 4]{Hinich:HAHA}. As Hinich pointed out, to see that this assignment respects maps of operads, note that the right adjoint (the restriction functor) preserves underlying chain complexes, and thus preserves fibrations and weak equivalences in $\chain$. 

The statement about quasi-isomorphisms is implied by the proof of \cite[Theorem 4.7.4]{Hinich:HAHA}.
\end{proof}
\begin{remark}\label{rem:boundedHinich}
Hinich explicitly works over unbounded complexes. Following Remark~\ref{rem:bounded}, to transfer a model structure along the free $\Pp$-algebra functor from complexes in degree at least $m$ to $\Pp$-algebras in degree at least $m$, there is a homological condition on the cofibrant generators that must be verified~(see \cite{Crans:QCMSS} or \cite[1.4.23]{CisinskiPCMTH}). The cofibrant generators in the category of complexes in degree at least $m$ are a subset of the cofibrant generators in the category of unbounded complexes. As a result, to extend Hinich's result to the category of bounded complexes, it suffices to note that the inclusion functor from $\Pp$-algebras in degree at least $m$ into unbounded $\Pp$-algebras is a left adjoint which reflects quasi-isomorphisms.  
\end{remark}

Proposition \ref{prop:arrowcocomplete} gives a functor from the undercategory of $\Omega \Cc$ in operads to the arrow category of $\cocomplete$. Theorem \ref{operadmodel} gives a functor from $\operads$ to $\model$. We will find a lift of the functor from Proposition \ref{prop:arrowcocomplete} that lands in model categories. To wit:
\begin{thm}\label{thm:lift}
Let $\alpha$ and $\beta$ be twisting morphisms $\Omega\Cc\to\Pp$ and let $f$ be a map of twisting morphisms $\alpha\to\beta$. The following is true.
\begin{enumerate} 
\item  There is a model category structure on $\Cc\coalg$, called the $\alpha$-model structure, such that the cobar functor $\Omega_\alpha$ is a left Quillen functor to $\Pp\alg$. 
\item
The weak equivalences and cofibrations of the $\alpha$-model structure are created by $\Omega_\alpha$. In fact, the cofibrations are also created by the forgetful functor to $\chain$ and are thus independent of $\alpha$.
\item
 The identity on $\Cc\coalg$ is a left Quillen functor (in fact a left Bousfield localization), from the $\alpha$-model structure to the $\beta$-model structure. This functor, $\Omega_\alpha$, and $\Omega_\beta$ commute with the left Quillen functor given by induction along $f$ between the categories of algebras over the operads.
 \end{enumerate}
\end{thm}
The proof of Theorem~\ref{thm:lift} will be deferred to Section~\ref{sec:proof}.

We use the notation $\Cc\coalg_\alpha$ for $\Cc\coalg$ equipped with the $\alpha$-model structure.
\begin{remark}
A clean encapsulation of the first part of Theorem~\ref{thm:lift} is the existence of a filler for the following diagram:

\begin{center}
\begin{tikzcd}
{}& \arr(\cocomplete)\\
\twist_{\Cc} = \Omega\Cc\downarrow\operads
\ar{ur}{\mathbf{\Omega}}\ar[dotted]{r}
\ar{d}[swap]{\text{target}}
& \arr(\model)\ar{u}[swap]{\mathrm{forget}}\ar{d}{\mathrm{target}}
\\
\operads\ar{r}
& \model.
\end{tikzcd}
\end{center}

This diagram does not show that the weak equivalences and cofibrations are created by $\Omega_\alpha$ but otherwise expresses the connection among the different extant structures and constructions.

The third part of Theorem~\ref{thm:lift} says that the commutativity of the left diagram in $\operads$ implies the commutativity of the right diagram in $\model$. This upgrades the implication of diagrams described in Proposition~\ref{prop:arrowcocomplete}.
\begin{center}\begin{tikzcd}
\Omega\Cc
\ar{d}[swap]{\id}
\ar{r}{\alpha}
&\Oo\ar{d}{f}
\\
\Omega\Cc\ar{r}[swap]{\beta}&\Pp
\end{tikzcd}
\qquad
\begin{tikzcd}
\Cc\coalg_\alpha
\ar{d}[swap]{\id}\ar{rr}{\Omega_\alpha} &&
\Oo\alg
\ar{d}{\text{induction along }f}
\\
\Cc\coalg_\beta\ar{rr}[swap]{\Omega_\beta}&&
\Pp\alg.
\end{tikzcd}
\end{center}
\end{remark}
The Vallette model structure and the Aubry--Chataur/Hess--Shipley model structure (in the case of coalgebras over a cooperad) are recovered by Theorem~\ref{thm:lift}.
\begin{cor}\label{initialterminalstructures}
The value of the functor of Theorem~\ref{thm:lift} on the initial object of $\twist_\Cc$ is the Vallette model structure of Theorem~\ref{Vallettemodel}. The value of the functor on the terminal object of $\twist_\Cc$ is the Aubry--Chataur/Hess--Shipley model structure, where weak equivalences and cofibrations are created by forgetting to $\chain$ (explicitly, they are quasi-isomorphisms and inclusions).
\end{cor}
\begin{proof}
The Vallette model structure has weak equivalences created by $\Omega_{\id}$ and cofibrations created by the forgetful functor to $\chain$. Thus the constructions yield the same model structure. 

For the terminal twisting morphism $\aug$, the cobar functor $\Omega_\aug$ is the forgetful functor to $\chain$.
\end{proof}

The construction of model categories from twisting morphisms is functorial not only in operads but also in cooperads. That is, if $f:\Cc\to \Dd$ is a map of cooperads, then any twisting morphism $\alpha:\Omega\Dd\to \Pp$ induces a twisting morphism $\alpha' :=\alpha\circ\Omega f:\Omega\Cc\to \Pp$. 
\begin{prop}Let $f:\Cc\to \Dd$ be a map of cooperads and $\alpha:\Omega\Dd\to\Pp$ a twisting morphism. Then pushforward $f_*$ is a left Quillen functor from $\Cc\up{-coalg}$ with the $\Omega_{\alpha'}$ model structure to $\Dd\up{-coalg}$ with the $\Omega_\alpha$ model structure.
\end{prop}
\begin{proof}
The pushforward functor does not change the underlying chain complex. Therefore, since cofibrations are created in $\chain$, they are preserved by pushforward. Next, note that $\Omega_{\alpha'} = \Omega_\alpha \circ f_*$. Let $\phi$ be a map in $\Cc\up{-coalg}$ which becomes a quasi-isomorphism in $\Pp$-alg under this functor. Then $f_*\phi$ becomes a quasi-isomorphism in $\Pp$-alg under $\Omega_\alpha$. Therefore $f_*$ preserves weak equivalences.
\end{proof}
There is a simple criterion for the Quillen pair $(\Omega_\alpha, B_\alpha)$ to be a Quillen equivalence. 
\begin{prop} \label{prop:quillencriterion}
Let $\alpha:\Omega\Cc\to \Pp$ be a twisting morphism. If $\alpha$ is Koszul, then the Quillen pair
\[
\xymatrix{
\Cc\coalg_\alpha \ar@<1.4ex>[rr]^(.56){\Omega_\alpha}
&\perp&
\Pp\alg \ar@<1.4ex>[ll]^(.44){B_\alpha}
}
\] 
is a Quillen equivalence. If $\Pp$ is weight graded, the twisting morphism $\alpha$ respects weight gradings, and the characteristic is zero, then the converse is true.
\end{prop}
\begin{proof} 
Let $X$ range over all $\Cc$-coalgebras and $A$ over all $\Pp$-algebras. Consider the following conditions.
\begin{enumerate}
\item 
The pair $(\Omega_\alpha, B_\alpha)$ is a Quillen equivalence.
\item
A map $\Omega_\alpha X\to A$ is a weak equivalence if and only if the adjoint $X\to B_\alpha A$ is a weak equivalence.
\item
A map $\Omega_\alpha X\to A$ is a quasi-isomorphism if and only if $\Omega_\alpha X\to \Omega_\alpha B_\alpha A$ is a quasi-isomorphism.
\item
The map $\Omega_\alpha B_\alpha A\to A$ is a quasi-isomorphism.
\end{enumerate}
Conditions (1) and (2) are equivalent because every coalgebra is cofibrant and every algebra is fibrant. Conditions (2) and (3) are equivalent by the definitions of weak equivalence in the respective model categories. Fresse shows that when $\alpha$ is Koszul, condition (3) implies the condition (4)~\cite[4.2.4]{Fresse:OCCCOHMAOO}. With the additional assumptions, Loday--Vallette show the converse~\cite[Theorem 11.3.3]{LodayVallette:AO}.
\end{proof}
\begin{remark}
Consider a Koszul morphism $\alpha : \Omega \Cc \to \Pp$. We have the following diagram at the level of model categories
\[
\xymatrix{
\Pp\alg \ar@<1.4ex>[rr]^(.49){B_\alpha}
&\downvdash&
\Cc\coalg_\alpha \ar@<1.4ex>[ll]^(.51){\Omega_\alpha}
\ar@<1.4ex>[rr]^(.5){\id}
&\perp&
\Cc\coalg_\aug\ar@<1.4ex>[ll]^(.5){\id}.
}
\] 

There are several roughly equivalent ways to pass from a model category to an infinity category. One is to apply simplicial localization \cite{BarwickKan:RCMHTHT} followed by the homotopy coherent nerve $\mathscr{N}^{\textrm{hc}}$ \cite{Lurie:HTT}. If we apply this composite functor to our diagram, the leftmost and rightmost entry are well-known infinity categories:
\[
\xymatrix{
\mathscr{P}\textrm{-}\mathfrak{alg}
\ar@<1.4ex>[rr]^(.4){B_\alpha}
&\cong&
\mathscr{N}^{\textrm{hc}}\mathscr{L} \left(\Cc\coalg_\alpha \right)
 \ar@<1.4ex>[ll]^(.6){\Omega_\alpha}
\ar@<1.4ex>[rr]
&\perp&
\mathscr{C}\textrm{-}\mathfrak{coalg}_{\textrm{dp}}^{\textrm{nil}}
\ar@<1.4ex>[ll].
}
\] 

Here we use script lettering to denote infinity categories. We reemphasize on the right that we are working with conilpotent coalgebras and that when the characteristic is nonzero we get so-called the infinity category of divided-power coalgebras.

The composite $\infty$-functor $\mathscr{P}\textrm{-}\mathfrak{alg} \to \mathscr{C}\textrm{-}\mathfrak{coalg}_{\textrm{dp}}^{\textrm{nil}}$ is a left adjoint (as a composite of the equivalence $B_{\alpha}$ followed by a left adjoint) and can be identified with $\textrm{Bar}_{\Pp}^{\textrm{enh}}$, and similarly, the composite $\infty$-functor $\mathscr{C}\textrm{-}\mathfrak{coalg}_{\textrm{dp}}^{\textrm{nil}}$ is a right adjoint and can be identified with $\textrm{Cobar}_{\Pp}^{\textrm{enh}}$. 

The punchline is that the handedness of the dg bar and cobar constructions (for a Koszul morphism) is an artifact of Quillen equivalences coming in two kinds---left and right---while modeling equivalences of $\infty$-categories that have no such property. 
\end{remark}

\section{Examples}
The goal of this section is to provide simple, concrete examples of twisting morphisms inducing distinct model category structures on the same category of coalgebras.
\subsection{Coassociative coalgebras}
Let $\fieldk$ have characteristic zero. Consider the operad $\ass$~\cite[9.1]{LodayVallette:AO} governing associative algebras.\footnote{In most of chapter 9, this reference considers the nonsymmetric operad $As$ but essentially all the results hold with the symmetric operad $Ass$ replacing the nonsymmetric operad $As$.} Its linear dual $\ass^c$ is the cooperad governing coassociative coalgebras. We will use Theorem~\ref{thm:lift} to establish three distinct homotopy categories on ${\ass^c}\!\coalg$.

The homotopy theory of associative algebras is well-understood in terms of the Koszul twisting morphism 
\[\kappa:\Omega (\suspend \ass^c)\to \ass\]
where $\suspend$ is the operadic suspension. The Koszul twisting morphism $\kappa$ is determined by its image on generators. On generators, $\kappa$ takes the binary coproduct in $\ass^c$ to the binary product in $\ass$.

We will employ a desuspended version of this Koszul twisting morphism, also called $\kappa$:
\[\kappa:\Omega \ass^c \to \desuspend\ass.\]

The last sentence of Theorem~\ref{operadmodel} implies that because $\kappa$ is a Koszul twisting morphism, the $\kappa$-model structure on $\ass^c\!\coalg$ coincides with the initial $\id$-model structure. For ease we will use the former rather than the latter.

We will also consider the terminal $\aug$-model structure.

In addition, we will study a model structure arising from a twisting morphism $\beta$ constructed as follows (for more detail on the constituent ingredients of the construction see~\cite[13.1]{LodayVallette:AO}). 

There are maps of operads from $\lie\to\ass\to\com$. Here $\lie$ is the operad governing Lie algebras and $\com$ the operad governing commutative associative algebras. At the level of algebras, restriction along $\lie\to \ass$ gives the commutator Lie algebra of an associative algebra, while induction along $\ass\to\com$ is Abelianization. 

On the cooperad side, the linear dual map $\ass^c\to \lie^c$ induces a functor of the associated categories of coalgebras; a coassociative coalgebra can be skew-symmetrized to yield its {\em cocommutator Lie coalgebra}.

There is a Koszul twisting morphism, which we shall call $\kappa_\lie$, which makes the following square commute:
\[
\begin{tikzcd}
\Omega\ass^c\ar{r}{\kappa}\ar{d}&\desuspend\ass\ar{d}\\
\Omega\lie^c\ar{r}{\kappa_\lie}&\desuspend\com
\end{tikzcd}
\]
The third twisting morphism we will consider is the diagonal map of this square, which we will call $\beta$.

The three model structures on $\ass^c$-coalgebras induced by $id:\Omega \ass^c\to \Omega \ass^c$, $\beta:\Omega \ass^c\to \shiftedcom$, and $\aug:\Omega \ass^c\to \fieldk$ have the following weak equivalences.

\vspace{10pt}
\begin{tabular}{cp{7cm}}
twisting morphism & weak equivalences\\\hline
$\id$ or $\kappa$ & quasi-isomorphisms on coHochschild homology valued in the trivial comodule\\
$\beta$ & quasi-isomorphisms on coChevalley-Eilenberg homology of the cocommutator Lie coalgebra\\
$\aug$ & quasi-isomorphisms of coalgebras\\\hline
\end{tabular}
\vspace{10pt}

These three categories of weak equivalence are strictly nested. Let $\widetilde{C}$ denote the cofree conilpotent coalgebra on generators $x$ of degree $1$ and $y$ of degree $4$. In the following two examples, we have been lax about the signs involved in some differentials; the results are true no matter which signs are correct.
\begin{example}
Let $X$ be the subcoalgebra of $\widetilde{C}$ spanned by $x$ and $w\coloneqq x\otimes x$, equipped with the differential $w\mapsto x$. The inclusion of the zero coalgebra into $X$ is a quasi-isomorphism, so $X$ is equivalent to the zero coalgebra in the $\aug$-model structure. 

However, we shall see that $\Omega_\beta X$ has nontrivial homology, which will show that $X$ is not equivalent to the zero coalgebra in the $\beta$-model structure.

The underlying vector space of $\Omega_\beta X$ is $\desuspend\com X\cong \com(X[-1])[1]$. Let the desuspension be denoted by a tilde, then this is $\fieldk[\tilde{x},\tilde{w}][1]$, where $|\tilde{x}|=0$ and $|\tilde{w}|=3$. Thus
\begin{eqnarray*}
(\Omega_\beta X)_2&\cong& \tilde{w}\fieldk[\tilde{x}][1]
\\(\Omega_\beta X)_1&\cong& \fieldk[\tilde{x}][1]
\\(\Omega_\beta X)_0&\cong& 0.
\end{eqnarray*}
The differential takes $\tilde{w}\tilde{x}^n$ to $\pm 2\tilde{x}^{n+2} \pm \tilde{x}^{n+1}$. This implies that the element $\tilde{x}[1]$ is not in the span of the image of this differential and thus survives in homology.
\end{example}
\begin{example}
Let $C_1$ be the subcoalgebra of $\widetilde{C}$ spanned by $x$. Let $C_2$ be the subcoalgebra of $\widetilde{C}$ spanned by $x$, $y$, and $z\coloneqq x\otimes y + y\otimes x$, equipped with the differential $z\mapsto y$. We will consider the inclusion $C_1\to C_2$.

The cocommutator Lie coalgebras of both $C_1$ and $C_2$ have vanishing cobracket. Then (as in the previous example) $\Omega_\beta C_i$ is the free shifted commutative algebra on the underlying chain complex of $C_i$.  For any operad $\Pp$ satisfying our mild conditions (including $\desuspend\com$), the free $\Pp$-algebra functor is a left Quillen functor from $\chain$ to the model category of $\Pp$-algebras so it preserves trivial cofibrations, in particular taking them to quasi-isomorphisms. In chain complexes, the inclusion $C_1\to C_2$ is a trivial cofibration, and thus $\Omega_\beta C_1\to \Omega_\beta C_2$ is a quasi-isomorphism and the inclusion $C_1\to C_2$ is a weak equivalence in the $\beta$-model structure.

This fact can also be checked by hand or one could note that on the cocommutator Lie coalgebra, the inclusion is a filtered quasi-isomorphism (see~\cite{Vallette:HTHA}) and thus becomes a quasi-isomorphism under the Koszul twisting morphism $\lie^c\to \desuspend\com$. 

On the other hand, we shall see that applying $\Omega_{\kappa}$ to the inclusion $C_1\to C_2$ yields algebras with different homology. In particular, $\Omega_{\kappa}C_1$, the free $\desuspend\ass$-algebra on one generator in degree $1$, is concentrated in degree $1$. 

Now consider $\Omega_{\kappa}C_2$. The underlying vector space of this $\desuspend\ass$-algebra can be written as $T(C_2[-1])[1]$, where $T$ denotes the tensor algebra, so that
\begin{eqnarray*}
(\Omega_\kappa C_2)_5&\cong& \langle \tilde{x}^m\tilde{z}\tilde{x}^n[1]\rangle
\\(\Omega_\kappa C_2)_4&\cong& \langle \tilde{x}^m\tilde{y}\tilde{x}^n[1]\rangle
\\(\Omega_\kappa C_2)_3&\cong& 0
\end{eqnarray*}
The differential takes $\tilde{x}^m\tilde{z}\tilde{x}^n[1]$ to
$\pm \tilde{x}^m\tilde{y}\tilde{x}^n[1] \pm \tilde{x}^{m+1}\tilde{y}\tilde{x}^n[1]\pm \tilde{x}^m\tilde{y}\tilde{x}^{n+1}[1]$. The element $\tilde{y}[1]$ is not in the span of the image of this differential and thus survives in homology.
\end{example}

\subsection{An elementary example in comodules}\label{subsection:comodules}
In this subsection, let all cooperads and operads be concentrated in arity one. Then a (co)augmented (co)operad can be identified with a (co)unital (co)associative (co)algebra. All of the constructions we have used have classical cognates under this identification. A twisting morphism $\alpha:\Omega\Cc\to \Pp$ is called a \emph{twisting cochain} between the coalgebra $\Cc$ and the algebra $\Pp$. In this setting a ``(co)algebra'' over such a (co)operad is a (co)module over the (co)algebra. The adjoint functors $\Omega_\alpha$ and $B_\alpha$ induced by the twisting morphism $\alpha$ are the classical \emph{one-sided (co)bar} functors between $\Cc$-comodules and $\Pp$-modules. Theorem~\ref{thm:lift} applies just as well in this special case to construct a model category structure on $\Cc$-comodules. We include an example in this special case because the category of $\Cc$-comodules may be of general interest. 

Our example will exploit a fact about {\em completion functors}. Such functors accept a natural transformation from the identitity that is not generally a quasi-isomorphism. However, in some cases this natural transformation becomes a quasi-isomorphism under the appropriate bar functor. Morally, this is because a bar functor encodes derived indecomposables and completion does not affect indecomposables.

Consider the free counital coalgebra on variables $\mu$ and $\nu$ in degree $1$, and let $\Cc$ be the subcoalgebra spanned by $1$, $\mu$, $\nu$, and $\eta\coloneqq\mu\otimes \nu-\nu\otimes \mu$. This is a coalgebraic version of the exterior algebra on $\mu$ and $\nu$. We will employ the twisting morphisms determined by the following data:
\begin{eqnarray*}
\Omega\Cc\xrightarrow{\kappa} \fieldk[x,y]&\text{via}& s^{-1}\mu\mapsto x,\quad s^{-1}\nu\mapsto y;\\
\Omega\Cc\xrightarrow{\alpha} \fieldk[x]&\text{via}& s^{-1}\mu\mapsto x,\quad s^{-1}\nu\mapsto 0;\\
\Omega\Cc\xrightarrow{\aug} \fieldk&\text{via}& s^{-1}\mu\mapsto 0,\quad s^{-1}\nu\mapsto 0.
\end{eqnarray*}
Here, as in the case of coassociative coalgebras, $\kappa$ accepts a quasi-isomorphism from the initial twisting morphism $\id$ and can serve as a replacement for it.

We can use $B_\kappa$ to construct maps of comodules that demonstrate the difference among the corresponding three classes of weak equivalences.
\begin{example}
For $z$ an indeterminate, let $S(z)$ and $\widehat{S}(z)$ denote the kernel of the evaluation at zero $\fieldk[z]\to \fieldk$ and $\fieldk[[z]]\to \fieldk$, respectively. That is, these are non-unital polynomial and power series algebras in a single variable.

We can make $S(y)$ and $\widehat{S}(y)$ into $\fieldk[x,y]$-modules by letting $x$ act as $0$. Consider the map of comodules $B_\kappa S(y)\to B_\kappa \widehat{S}(y)$. This map of comodules is not a $\kappa$-weak equivalence because that would imply that $S(y)\to \widehat{S}(y)$ was a weak equivalence:
\[
\begin{tikzcd}
\Omega_\kappa B_\kappa S(y) \ar{r}\ar{d}{\simeq}
&
\Omega_\kappa B_\kappa \widehat{S}(y)\ar{d}{\simeq}
\\
S(y)\ar{r}&\widehat{S}(y).
\end{tikzcd}
\]
On the other hand, for any free $\fieldk[y]$-module $M$ concentrated in degree zero, the $\fieldk[x]$-module $\Omega_\alpha B_\kappa M$ is spanned by elements of the form $p\otimes q\otimes m$ where $p\in \fieldk[x]$, $q\in \Cc$, and $m\in M$. The differential is of the form:
\begin{eqnarray*}
d (p\otimes \eta\otimes m) &=& px\otimes \nu\otimes m - p\otimes \mu\otimes ym\\
d (p\otimes \mu\otimes m) &=& px \otimes 1\otimes m\\
d (p\otimes \nu\otimes m) &=& p\otimes 1\otimes ym.
\end{eqnarray*}
The homology of this $\fieldk[x]$-module is concentrated in degree zero, and is spanned by $1\otimes 1\otimes M/yM$. Thus the map $\Omega_\alpha B_\kappa S(y)\to \Omega_\alpha B_\kappa \widehat{S}(y)$ is an isomorphism on homology and so the map of comodules $B_\kappa S(y)\to B_\kappa \widehat{S}(y)$ is an $\alpha$-weak equivalence.
\end{example}
Essentially the same argument shows that the map $B_\kappa S(x)\to B_\kappa \widehat{S}(x)$ is a $\aug$-weak equivalence which is not an $\alpha$-weak equivalence.

\section{Proof of the main theorem}\label{sec:proof}
Recall that $\alpha$ is a twisting morphism $\Omega\Cc\xrightarrow{\alpha}\Pp$ which induces the adjunction
\[
\xymatrix{
\Cc\coalg_\alpha \ar@<1.4ex>[rr]^(.55){\Omega_\alpha}
&\perp&
\Pp\alg \ar@<1.4ex>[ll]^(.45){B_\alpha}.
}
\] 

In order to prove Theorem~\ref{thm:lift}, we will left induce along $\Omega_\alpha$ to get a model category structure on $\Cc$-coalgebras where the weak equivalences and cofibrations are created by $\Omega_\alpha$. 

Bayeh--Hess--Karpova--K\k{e}dziorek--Riehl--Shipley~\cite{BayehHessKarpovaKedziorekRiehlShipley:LIMSDC}, using a theorem of Makkai--Rosick\'y~\cite{MakkaiRosicky:CC}, give criteria for when this left induction is possible. 

We note the following elementary variation on Corollary~{2.21} in~\cite{BayehHessKarpovaKedziorekRiehlShipley:LIMSDC}.

\begin{prop}\label{prop:leftinduce}
Let $M$ be a locally presentable cofibrantly generated model category and $U:K\to M$ be a left adjoint whose domain is locally presentable. Define $U$-cofibrations and $U$-weak equivalences in $K$ to be created by $U$.

If there is a factorization of every morphism in $K$ into a $U$-cofibration followed by a $U$-weak equivalence, then this structure makes $K$ into a model category.
\end{prop}
\begin{proof}
Suppose a morphism $f:A\to C$ in $K$ lifts against all $U$-cofibrations. Factorize $f$ into a $U$-cofibration followed by a weak equivalence. Then $f$ fits into the following square:

\begin{center}
\begin{tikzcd}
A\ar[tail]{d}\ar[-, double equal sign distance]{r} & A\ar{d}{f} \\
B\ar{r}[swap]{\sim}\ar[dashed]{ur} & C.
\end{tikzcd}
\end{center}

Now the morphism $f$ is a retract of the $U$-weak equivalence $B\to C$. Then since $U$-weak equivalences are closed under retracts, $f$ is a $U$-weak equivalence. 

The condition that any map which lifts against all $U$-cofibrations be a $U$-weak equivalence is precisely the acyclicity condition of Theorem~2.23 of~\cite{BayehHessKarpovaKedziorekRiehlShipley:LIMSDC}. By that theorem, the left induced structure makes $K$ into a model category.
\end{proof}
\begin{lemma}\label{lemma:locallypresentable}
The categories of $\Pp$-algebras and $\Cc$-coalgebras are locally presentable.
\end{lemma}
\begin{proof}
The categories in question are respectively monadic and comonadic for a (co)monad whose underlying functor is of the form ${\displaystyle V \mapsto \bigoplus_{n \geq 0} M(n) \otimes_{\mathbb{S}_n} V^{\otimes n} }$ for some $\mathbb{S}$-module $M$. Such functors preserve sifted colimits, hence preserve filtered colimits, and hence are accessible. It follows from \cite[Corollary 2.47, Theorem 2.78]{AdamekRosicky:LPAC} that the category of $\Pp$-algebras is locally presentable and \cite[Proposition A.1]{ChingRiehl:CMCMC} that the category of $\Cc$-coalgebras is locally presentable.
\end{proof}

\begin{proof}[Proof of Theorem~\ref{thm:lift}]
We will use two results of Vallette~\cite{Vallette:HTHA}. Explicitly, he works in characteristic zero with a Koszul twisting morphism of the form $\Omega B \Pp\to \Pp$ for an operad $\Pp$, but working in arbitrary characteristic and with the Koszul twisting morphism $\id:\Omega \Cc\to \Omega \Cc$ for a cooperad $\Cc$ does not change his arguments.

Vallette constructs a factorization~\cite[2.5]{Vallette:HTHA} of an arbitrary map in $\Cc\coalg$ into a $\Omega_{\id}$-cofibration followed by an $\Omega_{\id}$-weak equivalence. Since induction is a left Quillen functor from $\Omega \Cc$-algebras to $\Pp$-algebras, it preserves cofibrations and weak equivalences between cofibrant objects. The image of a $\Cc$-coalgebra under $\Omega_{\id}$ in $\Omega \Cc$-algebras is always cofibrant. Since $\Omega_\alpha$ is the same as $\Omega_\id$ followed by induction, Vallette's factorization is in fact a factorization into an $\Omega_\alpha$-cofibration followed by an $\Omega_\alpha$-weak equivalence for any $\alpha$. This gives us the hypotheses of Proposition~\ref{prop:leftinduce}.

It remains only to show that cofibrations in the $\alpha$-model structure are created by forgetting to $\chain$. Let $f$ be a map of twisting morphisms $\alpha\to\beta$. Because induction along $f$ is a left Quillen functor, it preserves cofibrations, so the $\Omega_\alpha$-cofibrations are contained in the $\Omega_\beta$-cofibrations. Since $\id$ is initial and $\aug$ is terminal in $\twist_\Cc$, the $\Omega_\alpha$-cofibrations contain the $\Omega_\id$-cofibrations and are contained in the $\Omega_\aug$-cofibrations. Since $\Omega_\aug$ is literally the forgetful functor to $\chain$, it then suffices to show that any $\Cc\up{-coalg}$ map which becomes a cofibration under the forgetful functor to $\chain$ also becomes a cofibration under $\Omega_\id$. This is also shown by Vallette~\cite[Theorem~2.9, item~1]{Vallette:HTHA}. 
\end{proof}
\begin{remark}\label{rem:boundedVallette}
Vallette's construction is built using a factorization in unbounded algebras into a cofibration followed by a trivial fibration. To extend to the bounded case (see Remark~\ref{rem:bounded} and Remark~\ref{rem:boundedHinich}) we must argue that we can choose this factorization such that the middle algebra is in our bounded category. We instead use the factorization from the model category structure on bounded algebras. While there are bounded fibrations which are not unbounded fibrations, it is easy to check that bounded trivial fibrations are always unbounded fibrations.
\end{remark}

\begin{remark}
There are classical criteria for lifting model category structures along {\em right} adjoints roughly parallel to those of Proposition~\ref{prop:leftinduce}. The authors wondered whether it would be possible to dualize Theorem~\ref{thm:lift} and lift the $\Omega_{\aug}$ model structure on $\Cc$-coalgebras along the right adjoint $B_{\alpha}$ to create a model structure on $\Pp$-algebras where the weak equivalences and fibrations were created by $B_\alpha$. We have the model structure on $\Pp$-algebras that we expect to have the most restrictive weak equivalences (namely the Hinich model structure) just as we had in the Vallette model structure on $\Cc$-coalgebras. 

However, naively dualizing the argument of Theorem~\ref{thm:lift} does not work. In that proof, we use the induction functor, which is a left Quillen functor, to show that our factorization into an $\Omega_\id$ cofibration followed by an $\Omega_\id$ weak equivalence is also the same kind of factorization for $\Omega_\alpha$. This works because the image of $\Omega_\alpha$ is always $\Omega_\id$-cofibrant. In contrast, the image of $B_\alpha$ is not necessarily $B_\aug$-fibrant. In fact the coinduction functor that we would use in the dual goes ``the wrong way'' from $\Cc$-coalgebras to chain complexes. 
\end{remark}
\bibliographystyle{amsalpha} 
\bibliography{references-2015}
\end{document}